\documentclass[12pt,a4paper]{article}
\usepackage{euscript,amsfonts,amssymb,amsmath,amscd,amsthm,enumerate}

\usepackage{multicol}
\usepackage{graphicx}

\usepackage{color}

\usepackage{hyperref}

\newtheorem{theorem}{Theorem} 
\newtheorem*{theorem*}{Theorem}
\newtheorem{lemma}[theorem]{Lemma}
\newtheorem{definition}[theorem]{Definition}
\newtheorem{proposition}[theorem]{Proposition}

\theoremstyle{remark}

\newcommand{\GT}{\mathbb{GT}}

\newcommand{\GTC}{\widehat{\GT}}

\renewcommand{\H}{\mathcal H}

\renewcommand{\O}{\mathbb O}

\newcommand{\OO}{\mathcal O}

\newcommand{\E}{\mathbb E}

\renewcommand{\i}{\mathbf i}

\newcommand{\one}{\ensuremath{\mathbf 1}}
\newcommand{\zero}{\ensuremath{\mathbf 0}}
\newcommand{\mone}{\ensuremath{\mathbf {-1}}}

\title{From Alternating Sign Matrices to the Gaussian Unitary Ensemble}

\author{Vadim Gorin\thanks{Department of Mathematics, Massachusetts Institute of Technology, MA, USA and
 Institute for Information Transmission Problems of Russian Academy of Sciences,  Russia,
e-mail: vadicgor@gmail.com. The research was partially supported by by RFBR-CNRS grant
11-01-93105} }

\begin{document}

\maketitle

\begin{abstract}
The aim of this note is to prove that fluctuations of uniformly random alternating sign matrices
(equivalently, configurations of the six--vertex model with domain wall boundary conditions) near
the boundary are described by the Gaussian Unitary Ensemble and the GUE--corners process.
\end{abstract}

\section{Introduction}

\label{Section_intro_ASM}

An \emph{Alternating Sign Matrix} (ASM) of size $N$ is a $N\times N$ matrix whose entries are
either \zero, \one, or \mone, such that the sum along every row and column is $1$ and, moreover,
along each row and each column the nonzero entries alternate in sign, see Figure \ref{Fig_ASM} for
an example.

Since their introduction by Mills--Robbins--Rumsey \cite{MRR} ASMs attracted lots of attention
both in combinatorics and in mathematical physics. Enumerative properties of ASMs show their deep
connections with various classes of plane partitions and with a number of well-known lattice
models, see e.g.\ recent reviews in \cite{Z-Thesis}, \cite{Gier}, \cite[Introduction]{BFZ} and
references therein. Great interest to ASMs in statistical mechanics is related to the fact that
they are in bijection with configurations of the six-vertex model (or with square ice model) with
domain-wall boundary conditions as shown at Figure \ref{Fig_ASM}. A good review of the six--vertex
model can be found e.g.\ in the book \cite{Bax} by Baxter.

\begin{figure}[h]
\begin{multicols}{2}
\begin{center}
\Huge
$$
\begin{pmatrix}
  0& 0 & 0 &  1 & 0\\
  0& 1 & 0 & -1 & 1 \\
  0& 0 & 1 &  0 & 0\\
  1& 0 & 0 &  0 & 0\\
  0& 0 & 0 &  1 & 0
\end{pmatrix}
$$
\end{center}
\columnbreak
\begin{center}
 \scalebox{0.55}{\includegraphics{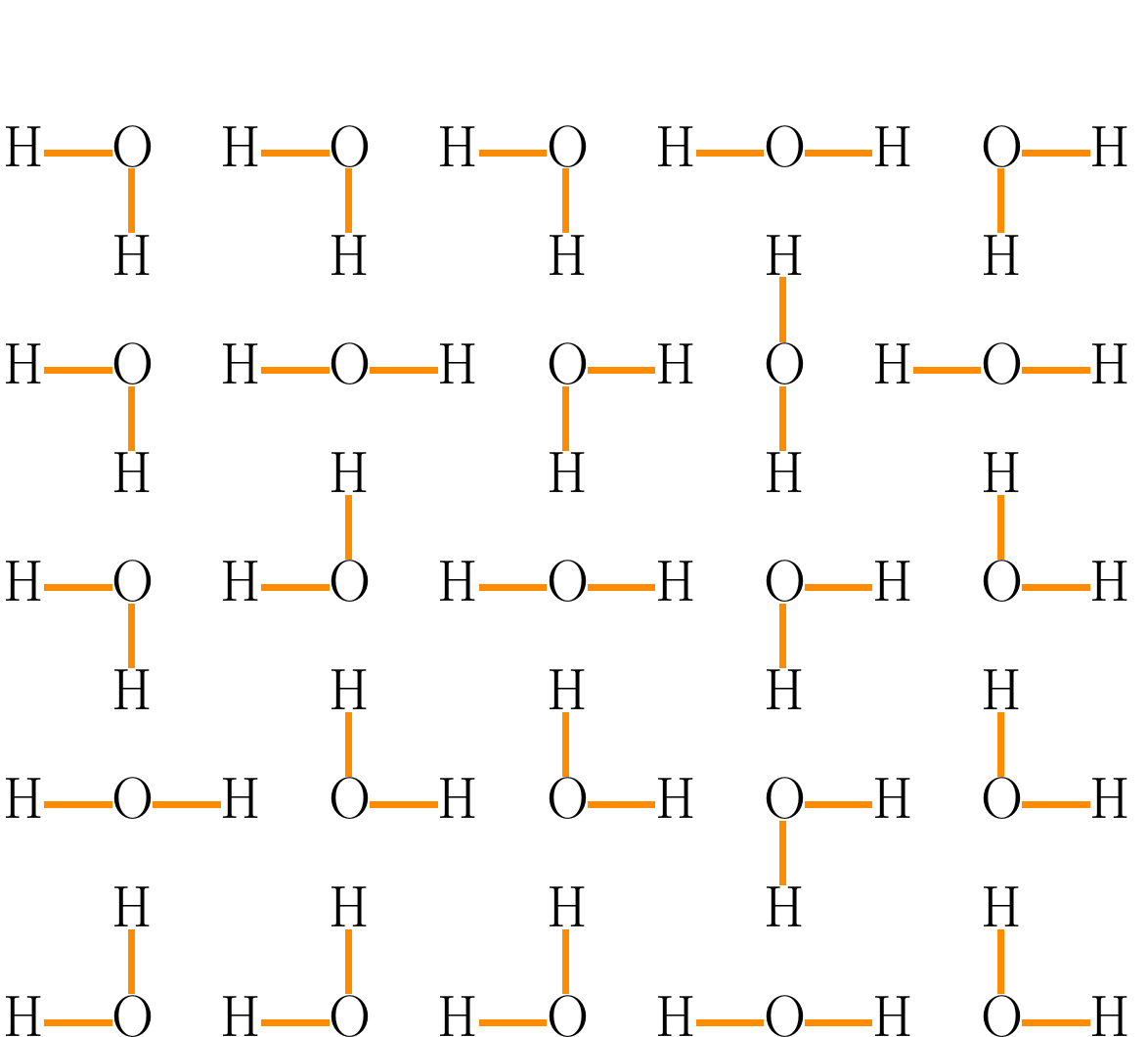}}
\end{center}
\end{multicols}
\caption{An alternating sign matrix of size 5 and the corresponding configuration of the
six--vertex  model (``square ice'')  with domain wall boundary condition. \one s in ASM correspond
to horizontal molecules H-O-H and \mone s to the vertical ones.\label{Fig_ASM}}
\end{figure}

\smallskip

Our interest in ASMs is probabilistic. We would like to know how a \emph{uniformly random} ASM of
size $N$ looks like when $N$ is large. The features of this model are believed to be similar to
the \emph{dimer models}, i.e.\ random lozenge tilings, plane partitions and domino tilings, cf.\
\cite{Kenyon} and also \cite{EKLP}, \cite{CLP}, \cite{J_nonintersecting}, \cite{KOS}, \cite{BGR},
\cite{Petrov-curves}. However, one of the key tools for studying the dimer models is the fact that
they can be described via \emph{determinantal point processes}. Such structure is not known for
uniformly random ASMs and one have to find different methods.

One of the (conjectural) features of uniformly random ASMs is the formation of the so--called
limit shape (also present in the dimer models), whose
 properties were studied by Colomo and Pronko
\cite{CP}; for the six--vertex model with more general boundary conditions the limit shape
phenomenon is discussed in \cite{PR}, \cite{R} (see also \cite{Z_shapes}). For ASMs the limit
shape theorem would claim, in particular, that when $N$ is large all non-zero matrix entries of a
uniformly random ASM of size $N$ lie with high probability inside a certain deterministic curve,
inscribed in $N\times N$ rectangle, see \cite{CP} for the details. As far as the author knows, the
exact form of this curve is still conjectural, but it closely matches the numeric simulations of
 \cite{AR}, \cite{SZ}.

Continuing the conjectural analogy with the dimer models, one expects various connections with
random matrices. In this article we study the asymptotic fluctuations of ASMs near the boundary of
the square and find such connection, which we now present.

Recall that the Gaussian Unitary Ensemble (GUE) of rank $N$ is the ensemble of random Hermitian
matrices $X=\{X_{ij}\}_{i,j=1}^N$ with probability density (proportional to) $\exp(-{\rm
Trace}(X^2))$ with respect to the Lebesgue measure. Let $\lambda^N_1\le
\lambda^N_2\le\dots\le\lambda^N_N$ denote the eigenvalues of $X$ and, more generally, for $1\le
k\le N$ let $\lambda^k_1\le \lambda^k_2\le\dots\le\lambda^k_k$ denote the eigenvalues of top--left
$k\times k$ corner $\{X_{ij}\}_{i,j=1}^k$ of $X$. The joint distribution of $\lambda_i^j$,
$i=1,\dots,j$, $j=1,\dots,N$ is known as the \emph{GUE--corners} process of rank $N$ (the name
GUE--minors process is also used, cf.\ \cite{JN}). The following theorem is the main result of the
present article.

\begin{theorem}
\label{theorem_ASM}
 Fix any $k$. \begin{enumerate}
 \item
 As $N\to\infty$ the probability that the number of \mone s in the first $k$ rows of a
 uniformly random ASM of size $N$ is maximal possible (i.e.\ there is one \mone\, in the second row, two \mone s in the third row, etc)  tends to
 $1$, and, thus, there are $k(k-1)/2$ interlacing \one s in the first $k$ rows with high probability.
 \item
 Let $\eta(N)^i_j$,
 $i=1,\dots,j$, $j=1,\dots,k$ denote the column number of the $i^{th}$ \one\, in the $j^{th}$ row of the uniformly
 random ASM, where we agree that $\eta(N)^i_j=+\infty$ if there are less than $i$ \one s in the $j^{th}$
 row. Then the random vector
\begin{equation}
\label{eq_scaled_coordinates}
 \sqrt{\frac{8}{3N}} \Big(\eta(N)^i_j-N/2\Big)
\end{equation}
 weakly converges to the
 GUE--corners process as $N\to\infty$.
\end{enumerate}
\end{theorem}
\noindent {\bf Remark.} Symmetries of uniformly random ASMs imply an analogue of Theorem
\ref{theorem_ASM} for the last $k$ rows, first $k$ columns and last $k$ columns of ASM. It is very
plausible that the four limiting GUE--corners processes are jointly independent.

\medskip

A number of results similar to Theorem \ref{theorem_ASM} for models of random Young diagrams and
random tilings related to the determinantal point processes is known, see \cite{Bar}, \cite{JN},
\cite{OR}, \cite{Nord}, \cite{GS}, \cite{GP}. Moreover, for random lozenge tilings the
GUE--corners process is believed to be the \emph{universal scaling limit} near an edge of the
boundary of the tiled domain, cf.\ \cite{OR}, \cite{JN}, \cite{GP}. Interestingly, the number of
ASMs is the same as the number of lozenge tilings of a hexagon with certain symmetries (see e.g.\
\cite{BP} and references therein.) However, this fact remains quite mysterious and no bijective
proof of it is known; Theorem \ref{theorem_ASM}, thus, gives another indication that direct
combinatorial connection between ASMs and lozenge tilings should exist.

\smallskip

Theorem \ref{theorem_ASM} was conjectured in \cite{GP}, in the same paper a partial result towards
Theorem \ref{theorem_ASM} was proved. Our argument relies on this result, so let us present it.

Let $\Psi_k(N)$ denote the sum of coordinates of \one s minus the sum of coordinates of \mone s in
the $k^{th}$ row of the uniformly random ASM of size $N$. In \cite{GP} it is proved that the
centered and rescaled random variables $\Psi_k(N)$ converge to the collection of i.i.d.\ Gaussian
random variables as $N\to\infty$.

\begin{theorem}[Theorem 1.10 in \cite{GP}]
\label{Theorem_ASM_GP}
 For any fixed $k$ the random variable $\sqrt{\frac83} \frac{\Psi_k(N)-N/2}{\sqrt{N}}$ weakly converges to the standard normal
 random variable $N(0,1)$. Moreover, the joint distribution of any collection of such variables converges to the
 distribution of independent standard normal random variables.
 \end{theorem}

We believe (but we do not have a proof) that an analogue of Theorem \ref{theorem_ASM} should hold
 for more general measures on ASMs. A natural class of measures can be obtained through the
 correspondence with six--vertex model. In the latter model one typically subdivides 6 types of
 vertices into 3 groups and assigns weights $a$, $b$, $c$ to these three groups. The
 probability of a configuration is further set to be proportional to the product of the weights of
 its vertices. For instance, these are the settings of the celebrated Izergin--Korepin formula
\cite{I}, \cite{Kor} for the partition function of the six--vertex model with domain wall boundary
 conditions. Asymptotics of this partition function in the limit regime which is somewhat similar to the one used in
  arguments of \cite{GP} (leading to Theorem \ref{Theorem_ASM_GP}) was also investigated in
  \cite[Appendix B]{CP2}, \cite[Appendix]{CPZ}.

 For one particular choice of the parameters $a$, $b$ and $c$ known as  ``the free fermion point'' of
 the six--vertex model an analogue of Theorem \ref{theorem_ASM} follows from the results of \cite{JN}. In terms
 of the ASMs this choice of weights corresponds to assigning the probability proportional to
 $2^{n_1}$ to an alternating sign matrix with $n_1$ \one s. This case is closely related to uniformly random domino
 tilings of the Aztec diamond (as is explained in \cite{EKLP}, \cite{FS}), to Schur measures (see \cite{BG} for a recent review)
 and to determinantal point processes, which makes it somewhat simpler.

 \medskip

 In the rest of the article we provide a proof of Theorem \ref{theorem_ASM}, which is organized as follows. In Section
 \ref{Section_Gelfand_Tsetlin_patterns} we study various classes of Gelfand--Tsetlin patterns and
 Gibbs measures on them. In Section \ref{Section_tightness} we prove that the distribution of random vector
 \eqref{eq_scaled_coordinates} is tight as $N\to\infty$. In Section \ref{Section_proof} we combine
 all the obtained results to finish the proof.

\section{Gibbs Measures on Gelfand--Tsetlin patterns}

\label{Section_Gelfand_Tsetlin_patterns}

\subsection{Half-Strict Gelfand--Tsetlin patterns}

Let $\GT_N$ denote the set of $N$--tuples of \emph{distinct} integers:
\begin{equation}
\label{eq_strict_pattern}
 \GT_N=\{\lambda\in\mathbb{Z}^N\mid \lambda_1<\lambda_2<\dots<\lambda_N\}.
\end{equation}
We say that $\lambda\in\GT_N$ and $\mu\in\GT_{N-1}$ \emph{interlace} and write $\mu\prec\lambda$ if
\begin{equation}
\label{eq_interlace}
 \lambda_1\le \mu_1 \le \lambda_2 \le\dots\le \mu_{N-1}\le\lambda_N.
\end{equation}
Note that the inequalities in \eqref{eq_strict_pattern} are strict, while in \eqref{eq_interlace}
they are weak.

Let $\GT^{(N)}$ denote the set of sequences
$$
 \mu^1\prec\mu^2\prec\dots\prec\mu^N,\quad\quad \mu^i\in\GT_i,\, 1\le i\le N,\quad
 \mu^i\prec\mu^{i+1},\, 1\le i<N.
$$
We call the elements of $\GT^{(N)}$ \emph{half--strict} Gelfand--Tsetlin patterns\footnote{The
name comes from the fact that an analogous object when all the inequalities are not strict is
closely related to the representations of unitary groups and Gelfand--Tsetlin basis in such
irreducible representations} (they are also known as monotonous triangles, cf.\ \cite{MRR}).

 For $\lambda\in\GT_N$, let $\GT^{(N)}_\lambda\subset \GT^{(N)}$ denote the set of half--strict
 Gelfand--Tsetlin patterns $\mu^1\prec\dots\prec\mu^N$ such that $\mu^N=\lambda$.

\begin{lemma} \label{lemma_ASM_and_pattern} The set of ASMs of size $N$ is in bijection with $\GT^{(N)}_{1<2<\dots<N}$. The
bijection is given by
$$
 ASM=(r^1,\dots,r^N)\,\longmapsto\, \mu^1\prec\mu^2\prec\dots\prec\mu^N\in \GT^{(N)}_{1<2<\dots<N},
$$
where $\mu^k$ encodes the column numbers of \one s in the sum of the first $k$ rows
$r^1+\dots+r^k$ of an ASM.
\end{lemma}
\begin{proof} This is straightforward, see also \cite{MRR}.
\end{proof}
 Under the above identification, the random variables $\Psi_k(N)$ of Theorem
\ref{Theorem_ASM_GP} turn into the differences
$$
 \Psi_k(N)=|\mu^k|-|\mu^{k-1}|,
$$
where $|\mu^k|$ is the sum of coordinates $\mu_1^k+\dots+\mu_k^k$ of $\mu^k\in\GT_N$, and
$\mu^1\prec\mu^2\prec\dots\prec\mu^N$ is the uniformly random element of
$\GT^{(N)}_{1<2<\dots<N}$.

\begin{definition} \label{def_discrete_gibbs} A probability measure $\rho$ on $\GT^{(k)}$ is called \emph{Gibbs measure} if
for  any $\lambda\in\GT_k$, the restriction of $\rho$ on $\GT^{(k)}_\lambda$ is proportional to
the uniform distribution on $\GT^{(k)}_\lambda$:
$$
 \rho\Big|_{\GT^{(k)}_\lambda} = \rho_k(\lambda) \cdot \text{Uniform measure on } \GT^{(k)}_\lambda,
$$
where $\rho_k(\cdot)$ is the projection of $\rho$ on $\GT_k$.
\end{definition}
Clearly, if $\mu^1\prec\mu^2\prec\dots\prec\mu^N\in \GT^{(N)}_{1<2<\dots<N}$ corresponds to
uniformly random ASM as in Lemma \ref{lemma_ASM_and_pattern}, then for any $1\le k\le N$, the
distribution of $\mu^1\prec\mu^2\prec\dots\prec\mu^k$ is a Gibbs measure on $\GT^{(k)}$.

\subsection{Continuous Gibbs property}

Let us introduce a continuous analogue of the set of half-strict Gelfand--Tsetlin patterns
$\GT^{(N)}$.

Let $\GTC_N$ denote the set of $N$--tuples of reals:
\begin{equation}
\label{eq_continuous_pattern}
 \GTC_N=\{\lambda\in\mathbb{R}^N\mid \lambda_1\le \lambda_2\le\dots\le\lambda_N\}.
\end{equation}
We say that $\lambda\in\GTC_N$ and $\mu\in\GTC_{N-1}$ \emph{interlace} and write $\mu\prec\lambda$
if
\begin{equation}
\label{eq_continuous_interlace}
 \lambda_1\le \mu_1 \le \lambda_2 \le\dots\le \mu_{N-1}\le\lambda_N.
\end{equation}
Let $\GTC^{(N)}$ denote the set of sequences
$$
 \mu^1\prec\mu^2\prec\dots\prec\mu^N,\quad\quad \mu^i\in\GT_i,\, 1\le i\le N,\quad
 \mu^i\prec\mu^{i+1},\, 1\le i<N.
$$
We call the elements of $\GTC^{(N)}$ \emph{continuous} Gelfand--Tsetlin patterns.

 For $\lambda\in\GTC_N$, let $\GTC^{(N)}_\lambda\subset \GTC^{(N)}$ denote the set of continuous
 Gelfand--Tsetlin patterns $\mu_1\prec\dots\prec\mu_N$ such that $\mu_N=\lambda$.

The following definition is a straightforward analogue of Definition \ref{def_discrete_gibbs}.
\begin{definition}
\label{def_cont_gibbs} A probability measure $\rho$ on $\GTC^{(k)}$ is called \emph{Gibbs measure}
if for any $\lambda\in\GTC_N$, the conditional distribution of $\rho$, given that $\mu^N=\lambda$
is the uniform distribution on on $\GTC^{(k)}_\lambda$, i.e.
$$
 \rho(\cdot\mid \mu^k=\lambda) = \text{Uniform measure on } \GTC^{(k)}_\lambda.
$$
\end{definition}

For $N$--tuple
$\lambda=(\lambda_1\le\lambda_2\le\dots\le\lambda_N)\in\GTC_N$ set
$$
|\lambda|=\lambda_1+\lambda_2+\dots+\lambda_N.
$$

\begin{proposition}
\label{proposition_characterization_of_GUE}
 Let $\rho$ be a Gibbs measure on $\GTC^{(N)}$ and let
 $\mu^1\prec\mu^2\prec\dots\prec\mu^N$ be $\rho$--distributed random
 element of $\GTC^{(N)}$. Suppose that
 $$
  |\mu^1|,\, |\mu^2|-|\mu^1|,\, |\mu^3|-|\mu^2|,\, \dots,
  |\mu^N|-|\mu^{N-1}|
 $$
 is a Gaussian vector with i.i.d.\ $N(0,1)$--distributed components.
 Then $\rho$ is the GUE--corners process of rank $N$.
\end{proposition}
\begin{proof}
 Let $\H(N)$ denote the set of $N\times N$ Hermitian
 matrices and let $U(N)$ denote the group of all $N\times N$ unitary
 matrices. Note that $U(N)$ acts on $\H(N)$ by conjugations and this action
 preserves eigenvalues of Hermitian matrices. Take any
 $\lambda\in\GTC_N$, let $X(\lambda)$ denote the diagonal matrix
 with eigenvalues $\lambda_1,\dots,\lambda_N$ and let $\O_\lambda$ denote
 the $U(N)$--orbit of $X(\lambda)$. Further, let $\OO_\lambda$ denote the
 \emph{orbital measure} on $\O_\lambda$, which is the pushforward of
 the (normalized) Haar measure on $U(N)$ with respect to the map
 $$
  U(N)\to \H(N),\quad u\mapsto  u X(\lambda) u^{-1}.
 $$
 Equivalently, if we view $\H(N)$ as the real Euclidian space of
 dimension $N^2$ equipped with norm $\|X\|^2={\rm Trace} (X^2)$, then
 $\OO_\lambda$ is merely a uniform measure on the orbit
 $\O_\lambda$.

 \smallskip

 Now let $\mu^1\prec\mu^2\prec\dots\prec\mu^N$ be distributed
 according to $\rho$ and let $\rho_N$ denote the measure on $\GTC_N$
 which is the projection of $\rho$ on $\mu^N$.

 Further let $\Theta_\rho$ denote the $U(N)$--invariant measure on
 $\H(N)$ which is $\rho_N$ mixture of the orbital measures
 $\OO_\lambda$. In other words, for any Borel set $\mathcal A\subset \H(N)$ we
 set
 $$
  \Theta_\rho(\mathcal A)=\int_{\GTC_N} \OO_\lambda(A) \rho_N(d\lambda).
 $$
 Suppose that $M=\{M_{ij}\}_{i,j=1}^N$ is a random
 $\Theta_\rho$--distributed Hermitian matrix. Define
 $\nu^k\in\GTC_k$, $k=1,\dots,N$, to be the eigenvalues of top--left
 $ k\times k$ corner of $M$, i.e.\ of $\{M_{ij}\}_{i,j=1}^k$.
 Straightforward linear algebra shows that
 $$
  \nu^1\prec\nu^2\prec\dots\prec\nu^N.
 $$
 We claim that the distribution of the vector $(\nu^k)$, $1\le k\le
 N$ is the same as that of $(\mu^k)$, $1\le k\le N$. Indeed, the
 distributions of $\mu^N$ and $\nu^N$ coincide by the construction.
 The conditional distribution of $(\mu^k)$, $1\le k \le N-1$ given
 $\mu^N$ is uniform, since $\rho$ is a Gibbs measure. The
 distribution of $(\nu^k)$, $1\le k \le N-1$ given
 $\nu^N$ is also uniform, which is a known property of orbital
 measures $\OO_\lambda$, see \cite{GN}, \cite[Proposition 4.7]{Bar}, \cite[Proposition 1.1]{Ner}.

\bigskip
 Now it remains to prove that $\Theta_\rho$ is GUE--distribution,
 i.e.\ its density with respect to Lebesgue measure is proportional to $\exp\big(-{\rm
 Trace}( X^2/2)\big)$. This is what we do in the rest of the proof.

Note that for $1\le k\le N$ we have
$$
M_{kk}={\rm Trace}\left(\{M_{ij}\}_{i,j=1}^k\right)-{\rm
Trace}\left(\{M_{ij}\}_{i,j=1}^{k-1}\right)=|\nu^k|-|\nu^{k-1}|.
$$
Therefore, $M_{kk}$ are i.i.d.\ standard Gaussians.

Further, the distribution of $M$ is uniquely defined by its Fourier transform $\phi$ (i.e.\
characteristic function), which is
$$
 \phi: \H(N)\to \mathbb C,\quad \phi(A)=\E\left( \exp(\i\cdot {\rm
 Trace}(AM))\right).
$$
Suppose that a $N\times N$ Hermitian matrix $A$ has eigenvalues $a_1\le a_2\le \dots\le a_N$ and
let ${\rm diag}(A)$ denote the diagonal matrix with the same eigenvalues, i.e.\ ${\rm
diag}(A)_{ij}=\delta_{ij}a_i$, $1\le i,j,\le N$. There exists $u\in U(N)$ such that $A=u {\rm
diag}(A) u^{-1}$. Using $U(N)$--invariance of the distribution of $M$ and the fact that ${\rm
Trace}(uBu^{-1})={\rm Trace}(B)$ for any matrix $B$, we get
\begin{multline}
\label{eq_Fourier_computation} \E\exp \Big(\i {\rm
 Trace}(AM)\Big)=\E\exp\Big(\i {\rm
 Trace}(u {\rm diag}(A) u^{-1}\cdot u M u^{-1})\Big)\\= \E\exp\Big(\i {\rm
 Trace}({\rm diag}(A) M)\Big)=\E\exp\left(\i\sum_{i=1}^N a_i
 M_{ii}\right)=\prod_{i=1}^N \exp\left(-\frac{(a_i)^2}{2}\right),
\end{multline}
where the last equality is the computation of the Fourier transform of the Gaussian distribution.
It remains to note that for the GUE--distribution,
 the Fourier transform is the same as the one given by \eqref{eq_Fourier_computation}.
\end{proof}

%
%
%
%
%
%
%
%

\section{Tightness}

\label{Section_tightness}

The aim of this section is to prove the following tightness
statement.

\begin{proposition}
\label{proposition_tightness}
 For $N=1,2,\dots$, let $\xi(N)=(\xi(N)^1\prec\xi(N)^2\dots\prec\xi(N)^N)$ be the uniformly random element of
 $\GT^{(N)}_{(1<2<\dots<N)}$. Then for any $k\ge 1$ the sequence of random variables $N^{-1/2}\big(\xi(N)^k-N/2\big)$,
 $N=1,2,\dots$ is tight (here $\xi(N)^k$ is the index, not power).
\end{proposition}

The proof of Proposition \ref{proposition_tightness} is based on the
following Lemma.

\begin{lemma}
\label{lemma_tightness}
 Fix $N>0$ and take a large enough positive number $L$.
 Let $\lambda\in\GT_N$ be such that
 $\lambda_N-\lambda_1=L$. Further suppose that
 $\mu^1\prec\dots\prec\mu^N$ is distributed according to the uniform measure on
 $\GT^{(N)}_\lambda$. Then for any $c\in\mathbb R$, we have
 \begin{equation}
 \label{eq_probability_estimate}
  {\rm Prob}\left( |\mu^1_1-c|>\frac{L}{2 N!}\right) \ge
  2^{-N-1},
 \end{equation}
\end{lemma}
Let us first use Lemma \ref{lemma_tightness} to prove Proposition \ref{proposition_tightness}.

\begin{proof}[Proof of Proposition \ref{proposition_tightness}]
 We argue by the contradiction.

 Suppose that random variables $N^{-1/2}\big(\xi(N)^k-N/2\big)$,
 $N=1,2,\dots$ are not tight as $N\to\infty$. Since any family of \emph{bounded} random variables
 on $\mathbb R^k$ is tight, this would imply that there exist a positive number $p>0$, a sequence
 of integers $N_1<N_2<N_3<\dots$ and a growing to $+\infty$ sequence $L_i$,
 $i=1,2,\dots$, such that
 $$
  {\rm Prob}\left( \sup_{j=1,\dots,k} \left|N_i^{-1/2}\big(\xi(N_i)^k_j-N_i/2\big)\right| > L_i\right) >p
 $$
 for every $i=1,2,3,\dots$. Since $\xi(N_i)^k_1< \xi(N_i)^k_2<\dots<\xi(N_i)^k_k$, one of the following three inequalities
 should then hold for infinitely many $i$s
 \begin{enumerate}
\item[(I)]  ${\rm Prob}\left(  N_i^{-1/2}\big(\xi(N_i)^k_1-N_i/2\big) >
L_i/2\right) >p/3$,
\item[(II)]  ${\rm Prob}\left(  N_i^{-1/2}\big(\xi(N_i)^k_k-N_i/2\big) <
 -L_i/2\right) >p/3$,
\item[(III)]  ${\rm Prob}\left(  N_i^{-1/2}\big(\xi(N_i)^k_k-\xi(N_i)^k_1\big) >
 L_i/2\right) >p/3$.
 \end{enumerate}
In case (I), due to interlacing conditions, ${\rm Prob}\left(
N_i^{-1/2}\big(\xi(N_i)^1_1-N_i/2\big) > L_i/2\right) >p/3$, which contradicts the convergence of
$N_i^{-1/2}\big(\xi(N_i)^1_1-N_i/2\big)$ to a Gaussian random variable, which is proved in Theorem
\ref{Theorem_ASM_GP}. Similarly, in case (II), ${\rm Prob}\left(
N_i^{-1/2}\big(\xi(N_i)^1_1-N_i/2\big) < -L_i/2\right) >p/3$, which again contradicts Theorem
\ref{Theorem_ASM_GP}.

In case (III) we note that the conditional distribution of $\xi(N_i)^a_b$, $b=1,\dots,a$,
$a=1,\dots,k-1$ given $\xi(N_i)^k=\lambda$ is the uniform measure on the set $\GT^{(N)}_\lambda$.
Then we can use Lemma \ref{lemma_tightness} and conclude that
$$
 {\rm Prob}\left(N_i^{-1/2}\big|\xi(N_i)^1_1-N_i/2\big| > \frac{L_i}{4k!} \right)\ge
 \frac{p}{3\cdot 2^{k+1}},
$$
which yet again contradicts Theorem \ref{Theorem_ASM_GP}.
\end{proof}

\begin{proof}[Proof of Lemma \ref{lemma_tightness}]
Induction in $N$.

First, suppose that $\lambda_{i+1}-\lambda_i\ge L/N$ for some $1<i<N-1$. Then the interlacing
condition $\mu^{N-1}\prec\mu^N=\lambda$ implies that (almost surely)
$\mu^{N-1}_{N-1}-\mu^{N-1}_1\ge L/N$. Then we can use the induction assumption which yields the
inequality \eqref{eq_probability_estimate}.

If $\lambda_{i+1}-\lambda_i<L/N$ for all $1<i<N-1$, then either
$\lambda_2-\lambda_1\ge L/N$ or $\lambda_N-\lambda_{N-1}\ge L/N$.
Without loss of generality we assume the latter.

Let us fix the values of $\mu^j_{j-1}$, $j=2,\dots,{N-1}$:
\begin{equation}
 \label{eq_conditioning}
 \mu^2_1=A_1,\quad \mu^3_2=A_2,\quad\dots, \mu^{N-1}_{N-2}=A_{N-2}
\end{equation}
Clearly, if we prove the inequality \eqref{eq_probability_estimate}
conditional on \eqref{eq_conditioning}, then the same inequality
would hold without conditioning.

Set also $\lambda_{N-1}=A_{N-1}$, $\lambda_N=B$. Note that
$$
 A_1\le A_2\le\dots\le A_{N-1} < B.
$$
 Now the distribution of $\mu^1_1$, $\mu^2_2$,\dots, $\mu^{N-1}_{N-1}$ is uniform on the set
defined by inequalities
\begin{equation}
\label{eq_domain}
 \mu^1_1\le\mu^2_2\le \dots\le \mu^{N-1}_{N-1}\le B,\quad \quad
 \mu^i_i \ge A_i,\, i=1,\dots,N-1
\end{equation}
 and also
\begin{equation}
 \mu_i^i> A_{i-1},\quad i=1,\dots,N-1.
\label{eq_domain2}.
\end{equation}
Note that when the numbers $A_i$ are distinct, then the inequalities \eqref{eq_domain2} are
automatically implied by \eqref{eq_domain}. On the other hand, if $A_i=A_{i+1}=\dots=A_{i+m}$,
then the inequalities for $\mu^{i+1}_{i+1}$,\dots, $\mu^{i+m}_{i+m}$ in \eqref{eq_domain} become
strict. Graphically, we can view the solutions to inequalities \eqref{eq_domain},
\eqref{eq_domain2} as $N-1$ points in $N-1$--rows of a Young diagram, as shown in Figure
\ref{Fig_inequalities}.
\begin{figure}[h]
\begin{center}
 \scalebox{0.8}{\includegraphics{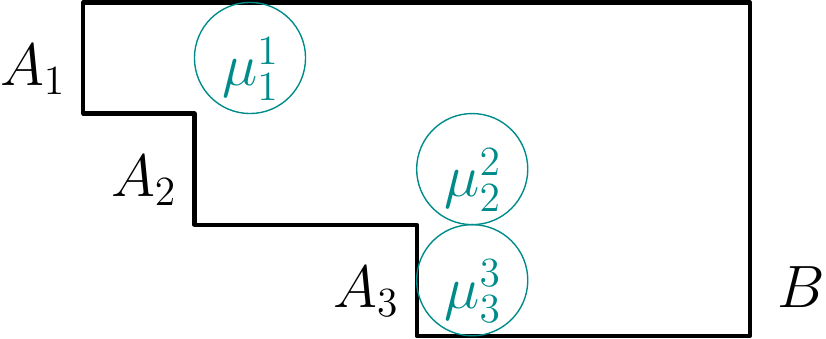}}
\end{center}
\caption{One of the solutions to inequalities \eqref{eq_domain}, \eqref{eq_domain2} represented as
points inside Young diagram. Here $N=4$\label{Fig_inequalities}}
\end{figure}
 From now on we assume that
all $A_i$ are distinct, the case of equal $A_i$s can be studied in the same way.

 Let $S(N-1;
A_1,\dots,A_{N-1}; B)$ denote the number of $(N-1)$--tuples ($\mu^1_1\le\mu^2_2\le \dots\le
\mu^{N-1}_{N-1}$) solving \eqref{eq_domain}, \eqref{eq_domain2}. The definition readily implies the
following monotonicity: if $A'_i\le A_i$, $i=1,\dots,N-1$ and $B'\ge B$, then
\begin{equation}
\label{eq_monoton} S(N-1; A_1,\dots,A_{N-1}; B)\le S(N-1; A'_1,\dots,A'_{N-1}; B'). \end{equation}
 Let us prove two
estimates:
 \begin{equation} \label{eq_est1}{\rm Prob}\left(\mu^1_1\le A_1 +\frac{B-A_1}{2^{N}}\right)\ge 2^{-N-1},\end{equation}
\begin{equation} \label{eq_est2}{\rm Prob}\left(\mu^1_1\ge A_1+ \frac{B-A_1}{2^{N-1}}\right)\ge 2^{-N-1}.\end{equation}
These two estimates together with observation that $B-A_1\ge B-A_{N-1}= L$ readily imply
\eqref{eq_probability_estimate}.

To prove \eqref{eq_est1} note that conditionally on $\mu^2_2,\dots,\mu^{N-1}_{N-1}$ the
distribution of $\mu_1^1$ (which arises from the uniform measure on the set defined by
inequalities \eqref{eq_domain}, \eqref{eq_domain2}) is uniform on the interval
$\{A_1,A_1+1,\dots,\mu^2_2\}$. Since $\mu^2_2\le B$, the desired inequality immediately follows.

To prove \eqref{eq_est2}, observe, first, that the distribution of $\mu^{N-1}_{N-1}$ is given by
\begin{equation}
\label{eq_proba}
 {\rm Prob}\left(\mu^{N-1}_{N-1}=k\right) = \frac{ S(N-2; A_1,\dots,A_{N-2};
 k)}{S(N-1;A_1,\dots,A_{N-1};B)},\quad k=A_{N-1}, A_{N-1}+1,\dots, B.
\end{equation}
The monotonicity property \eqref{eq_monoton} implies that the probability \eqref{eq_proba} is an
increasing function of $k$. Therefore,
\begin{equation}
\label{eq_x1}
 {\rm Prob}\left(\mu^{N-1}_{N-1}\ge \frac{A_{N-1}+B}2 \right) \ge \frac{1}{2}.
\end{equation}
Similarly studying the conditional distribution of $\mu^{N-2}_{N-2}$ given that
$\mu^{N-1}_{N-1}=k$, we get
\begin{equation}
\label{eq_x2}
 {\rm Prob}\left(\mu^{N-2}_{N-2}\ge \frac{A_{N-2}+k}2\,\Big |
 \, \mu^{N-1}_{N-1}=k\right) \ge\frac{1}{2}
\end{equation}
Combining \eqref{eq_x1} and \eqref{eq_x2} we conclude that
\begin{equation}
\label{eq_x3}
 {\rm Prob}\left(\mu^{N-2}_{N-2}\ge \frac{3A_{N-2}+B}4\right) \ge\frac{1}{2^2}.
\end{equation}
Further studying in the same way the conditional distribution of $\mu^{N-3}_{N-3}$ given
$\mu^{N-2}_{N-2}$ and $\mu^{N-1}_{N-1}$ and combing with \eqref{eq_x3} we get
\begin{equation}
 {\rm Prob}\left(\mu^{N-3}_{N-3}\ge \frac{7A_{N-3}+B}8\right) \ge\frac{1}{8}.
\end{equation}
Continuing this process, we finally get the inequality
\begin{equation}
 {\rm Prob}\left(\mu^{1}_{1}\ge \frac{(2^{N-1}-1)A_{1}+B}{2^{N-1}}\right) \ge 2^{1-N},
\end{equation}
which is \eqref{eq_est2}.
\end{proof}

\section{Proof of Theorem \ref{theorem_ASM}}

\label{Section_proof}

Proposition \ref{proposition_tightness} yields that centered and rescaled random variables
$\xi(N)^a_b$, $a=1,\dots,k$, $b=1,\dots a$ are tight as $N\to\infty$. Let $\zeta^a_b$ denote any
subsequential limit of the random vectors
\begin{equation}
\label{eq_rescaled_2}
  \sqrt{\frac{8}{3N}}\Big(\xi(N)^a_b-N/2\Big), \quad a=1,\dots,k,\quad b=1,\dots a.
\end{equation}
 Since the distribution of $\xi(N)^a_b$ for any $N$ satisfies the Gibbs property on $\GT^{(k)}$,
the distribution of $\zeta^a_b$ satisfies the (continuous) Gibbs property on $\GTC^{(k)}$. Now
combination of Proposition \ref{proposition_characterization_of_GUE} and Theorem
\ref{Theorem_ASM_GP} yields that the distribution of $\zeta$ is the GUE--corners process. Since
all the subsequential limits are the same, we conclude that \eqref{eq_rescaled_2} weakly converges
to the GUE--corners process.

In particular, this implies that with probability tending to $1$ all the coordinates of random
vector $\xi(N)^a_b$ become distinct as $N\to\infty$. This yields part 1 of Theorem
\ref{theorem_ASM}. Further, when the coordinates $\xi(N)^a_b$ are distinct, then
$\xi(N)^a_b=\eta(N)^a_b$, which finishes the proof of Theorem \ref{theorem_ASM}.


\begin{thebibliography}{EKLP}

\bibitem[AR]{AR} D.~Allison, N.~Reshetikhin, Numerical study of the 6--vertex model with
domain wall boundary conditions, Annales de l'institut Fourier, 55, no.\ 6 (2005), 1847--1869,
arXiv:cond-mat/0502314.

\bibitem[Bar]{Bar} Y.~Baryshnikov, GUEs and queues, Probability Theory and Related Fields, 119, no.\ 2 (2001),
256--274.

\bibitem[BG]{BG} A.~Borodin, V.~Gorin, Lectures on integrable probability,  arXiv:1212.3351

\bibitem[BGR]{BGR} A.~Borodin, V.~Gorin, E.~Rains, $q$-Distributions on boxed plane partitions. Selecta
Mathematica, New Series, 16 (2010), no.\ 4, 731--789, arXiv:0905.0679.

\bibitem[Bax]{Bax} R.~J.~Baxter, Exactly Solved Models in Statistical Mechanics, The Dover Edition, Dover, 2007.


\bibitem[BFZ]{BFZ} R.~E.~Behrend, P.~Di~Francesco, P.~Zinn--Justin, On the weighted enumeration of
Alternating Sign Matrices and Descending Plane Partitions, J. of Comb. Theory, Ser. A, 119, no.\ 2
(2012), 331--363. arXiv:1103.1176.

\bibitem[BP]{BP} D.~Bressoud, J.~Propp, How the Alternating
Sign Matrix Conjecture Was Solved, Notices of the American Mathematical Society, 46 (1999),
637--646.


\bibitem[CLP]{CLP} H.~Cohn, M.~Larsen, J.~Propp, The Shape of a Typical Boxed Plane Partition,
New York J. Math., 4 (1998), 137--165.  arXiv:math/9801059.

\bibitem[CP]{CP}  F.~Colomo, A.~G.~Pronko, The limit shape of large alternating sign matrices,
SIAM J.\ Discrete Math. 24 (2010), 1558-1571, arXiv:0803.2697.

\bibitem[CP2]{CP2} F.~Colomo, A.~G.~Pronko, The arctic curve of the domain-wall six-vertex model,
J. Stat. Phys. 138 (2010), 662--700.  arXiv:0907.1264.

\bibitem[CPZ]{CPZ}   F.~Colomo, A.~G.~Pronko, P.~Zinn-Justin,
The arctic curve of the domain wall six-vertex model in its antiferroelectric regime, J. Stat.
Mech. Theory Exp. (2010), no.\ 3.  arXiv:1001.2189.


\bibitem[EKLP]{EKLP} N.~Elkies, G.~Kuperberg,  M.~Larsen, J.~Propp, Alternating-sign matrices and domino tilings. I,II,
J. Algebraic Combin. 1 (1992), no. 2, 111--132; no. 3, 219-234. arXiv:math/9201305.



\bibitem[FS]{FS} P.~Ferrari, H.~Spohn, Domino tilings and the six-vertex model at its free fermion
point,  J. Phys. A: Math. Gen. 39 (2006), 10297--10306,  arXiv:cond-mat/0605406.


\bibitem[Gi]{Gier} J.~de~Gier, Fully packed loop models on finite geometries, Polygons, polyominoes and polycubes, Lecture
Notes in Physics, vol. 775, 2009, arXiv:0901.3963.

\bibitem[GS]{GS} V.~Gorin, M.~Shkolnikov, Limits of Multilevel TASEP and similar processes, to
appear in Annales de l'Institut Henri Poincar\'{e} (B) Probabilit\'{e}s et Statistiques,
arXiv:1206.3817.


\bibitem[GN]{GN} I.~M.~Gelfand, M.~A.~Naimark,
Unitary representations of the classical groups, Trudy Mat.~Inst.~Steklov, Leningrad, Moscow
(1950) (in Russian). (German transl.: Academie-Verlag, Berlin, 1957.)

\bibitem[GP]{GP} V.~Gorin, G.~Panova, Asymptotics of symmetric polynomials with applications to statistical mechanics and representation
theory, arXiv:1301.0634.

\bibitem[I]{I} A.~Izergin, Partition function of the six-vertex model in a finite volume, Sov. Phys.
Dokl. 32 (1987) 878--879.


\bibitem[J]{J_nonintersecting} K.~Johansson,  Non-intersecting Paths, Random Tilings and
Random Matrices.  Probab. Theory and Related Fields, 123 (2002), no. 2, 225--280,
arXiv:math/0011250.


\bibitem[JN]{JN} K.~Johansson and E.~Nordenstam. Eigenvalues of GUE minors. Electron. J. Probab.,
11, no. 50, 1342--1371 (electronic), 2006,  arXiv:math/0606760.



\bibitem[Ke]{Kenyon} R.~Kenyon, Lectures on dimers, IAS/Park City Mathematical Series, vol.\
16: Statistical Mechanics, AMS, 2009. arXiv:0910.3129.

\bibitem[KOS]{KOS} R.~Kenyon, A.~Okounkov, S.~Sheffield,
Dimers and Amoebae. Ann. Math. 163 (2006), no. 3, 1019--1056. arXiv:math-ph/0311005

\bibitem[Kor]{Kor} V.~Korepin, Calculation of norms of Bethe wave functions, Comm. Math. Phys. 86 (1982) 391--418.

\bibitem[MRR]{MRR} W.~H.~Mills, D.~P.~Robbins, and H.~Rumsey, Alternating sign matrices and
descending plane partitions, Journal of Combinatorial Theory, Series A, 34, no.\ 3 (1983):
340--359.

\bibitem[Ne]{Ner} Yu.~A.~Neretin, Rayleigh triangles and non-matrix interpolation of matrix beta
integrals, Sbornik: Mathematics (2003), 194(4), 515--540.

\bibitem[No]{Nord} E.~Nordenstam. Interlaced particles in tilings and random matrices. Doctoral thesis, KTH, 2009.


\bibitem[OR]{OR} A. Yu. Okounkov, N. Yu. Reshetikhin, The birth
of a random matrix, Mosc. Math. J., 6:3 (2006), 553--566


\bibitem[PR]{PR} K.~Palamarchuk, N.~Reshetikhin, The 6--vertex model with fixed boundary
conditions. Proceedings of Solvay Workshop ``Bethe Ansatz : 75 Years Later''. arXiv:1010.5011

\bibitem[P1]{Petrov-curves} L.~Petrov, Asymptotics of Random Lozenge Tilings via Gelfand-Tsetlin
Schemes, to appear in Probability Theorey and Related Fields, arXiv:1202.3901.

\bibitem[R]{R} N.~Reshetikhin, Lectures on the integrability of the 6-vertex model, in: J.~Jacobsen, S.~Ouvry,
V.~Pasquier, D.~Serban, L.~F.~Cugliandolo (Eds.) Exact Methods in Low-dimensional Statistical
Physics and Quantum Computing, Oxford University Press, Oxford 2010, pp. 197--266.
arXiv:1010.5031.

\bibitem[SZ]{SZ} O.~F.~Syljuasen, M.~B.~Zvonarev, Directed-loop Monte Carlo simulations of
vertex models. Phys. Rev. E 70, 016118 (2004), arXiv:cond-mat/0401491.

\bibitem[Z1]{Z_shapes} P.~Zinn-Justin, The influence of boundary conditions in the six-vertex model, arXiv:cond--mat/0205192

\bibitem[Z]{Z-Thesis} P.~Zinn-Justin, Six-vertex, loop and tiling models: integrability and combinatorics, Lambert Academic
Publishing, 2009, Habilitation thesis, arXiv:0901.0665.

\end{thebibliography}
\end{document}